\documentclass[12pt,a4paper,oneside]{amsart}
\usepackage{a4,a4wide}
\usepackage{amsfonts,amsmath,amssymb,amsthm,enumitem}

\usepackage[hyphens]{url}

\usepackage{mathtools}
\usepackage{esint}

\usepackage{tikz}
\usetikzlibrary{shapes,arrows}
\usetikzlibrary{intersections}

\usepackage{color}
\usepackage[latin1]{inputenc}
\numberwithin{equation}{section}

\usepackage{doi}

\usepackage{etoolbox}
\makeatletter
\patchcmd{\@maketitle}
  {\ifx\@empty\@dedicatory}
  {\ifx\@empty\@date \else {\vskip3ex \centering\footnotesize\@date\par\vskip1ex}\fi
   \ifx\@empty\@dedicatory}
  {}{}
\patchcmd{\@adminfootnotes}
  {\ifx\@empty\@date\else \@footnotetext{\@setdate}\fi}
  {}{}{}
\makeatother

\usepackage{tikz}
\usetikzlibrary{quotes,angles,positioning}

\newtheorem{theorem}{Theorem}[section]
\newtheorem{lemma}[theorem]{Lemma}
\newtheorem{proposition}[theorem]{Proposition}

\newtheorem*{theorem*}{Theorem}
\newtheorem*{lemma*}{Lemma}
\newtheorem*{proposition*}{Proposition}
\newtheorem*{corollary*}{Corollary}

\newcommand*{\dd}{\mathop{}\!\mathrm{d}}

\def\Lip{\mathrm{Lip}}

\def\EE{\mathbb{E}}
\def\PP{\mathbb{P}}

\def\P{\mathbb{P}}

\renewcommand{\phi}{\varphi}

\def\1{\mathbf{1}}

\def\XXint#1#2#3{{\setbox0=\hbox{$#1{#2#3}{\int}$ }
\vcenter{\hbox{$#2#3$ }}\kern-.57\wd0}}

\newcommand{\id}{\mathsf{id}}

\newcommand{\mres}{\mathbin{\vrule height 1.6ex depth 0pt width
0.13ex\vrule height 0.13ex depth 0pt width 1.3ex}}

\begin{document}
\title[From Combinatorics to PDEs]{From Combinatorics to Partial Differential Equations}

\author{Francesco Mattesini \address[Francesco Mattesini]{Universit\"at M\"unster \& Max Planck Institute for Mathematics in the Sciences Leipzig,  Germany} \email{mattesini.francesco@gmail.com}  \hspace*{0.5cm} Felix Otto \address[Felix Otto]{Max Planck Institute for Mathematics in the Sciences Leipzig, Germany} \email{Felix.Otto@mis.mpg.de} } 
\thanks{All authors are supported by the Deutsche Forschungsgemeinschaft (DFG, German Research Foundation) through the SPP 2265 {\it Random Geometric Systems}. FM has been funded by the Deutsche Forschungsgemeinschaft (DFG, German Research Foundation) under Germany's Excellence Strategy EXC 2044 -390685587, Mathematics M\"unster: Dynamics--Geometry--Structure. FM has been funded by the Max Planck Institute for Mathematics in the Sciences.  }

\begin{abstract}
The optimal matching of point clouds in $\mathbb{R}^d$ is a combinatorial problem;
applications in statistics motivate to consider random point clouds,
like the Poisson point process.
There is a crucial dependance on dimension $d$, 
with $d=2$ being the critical dimension.
This is revealed by adopting an analytical perspective, 
connecting e.\,g.~to Optimal Transportation. 
These short notes provide an introduction to the subject. 
The material presented here is based on a series of lectures held at the International 
Max Planck Research School during the summer semester 2022. Recordings of the lectures are available  at 
\url{https://www.mis.mpg.de/events/event/imprs-ringvorlesung-summer-semester-2022}.
\end{abstract}

\date{\today}
\maketitle
\tableofcontents
\section{Introduction}
Given a probability measure $\lambda$ and $N$ independent random samples $X_1, \dots, X_N$ of 
$\lambda$ we define the empirical measure
\begin{align}\label{fo01}
\mu^N := \frac1N \sum_{n=1}^N \delta_{X_n}. 
\end{align}
A natural statistical question is to ask how close the random measure (\ref{fo01})
is to the deterministic measure $\lambda$ for $N\gg 1$. 

\medskip
For the purpose of these notes we will slightly modify the question and investigate
the closeness between two independent copies of (\ref{fo01}). 
More precisely, we consider $2N$ independent samples $X_1, \dots, X_N, Y_1, \dots, Y_N$ from $\lambda$ 
and study how close the two empirical measures
\begin{align}\label{fo04}
\text{$\frac1N\sum_{n=1}^N \delta_{X_n}$ and $\frac1N\sum_{n=1}^N \delta_{Y_n}$}
\end{align}
are. This has the advantage that closeness can be restated in terms of the support of the two measures, 
i.\,e. distance between points. Indeed, for the empirical measures above it is natural to define 
the following distance
\begin{equation}\label{eq:matchingprob}
R:=\Big(\frac1N \min_{\sigma \in S_N} \sum_{n=1}^N |Y_{\sigma(n)} - X_n|^2\Big)^\frac{1}{2},
\end{equation}
where the minimum is taken over the set $S_N$ of all permutations of $\{1,\cdots,N\}$.
Clearly, (\ref{eq:matchingprob}) does not depend on the way the two sets of points are
enumerated and thus can be seen as a distance between the point clouds $\{X_1,\cdots,X_N\}$
and $\{Y_1,\cdots,Y_N\}$, and thus between the empirical measures (\ref{fo04}).
Any permutation $\sigma$ corresponds to a matching of the two point clouds;
(\ref{eq:matchingprob}) optimizes over matchings.
What is optmized in (\ref{eq:matchingprob}) is the sum of the {\it square} of the distances 
between matched points, and thus $R$ defines the square-averaged distance between matched points.

\medskip

The aim of these notes is to establish a classical result of the optimal matching theory, 
related to the asymptotics of the matching cost \eqref{eq:matchingprob} for $N\uparrow\infty$
in the simple case when $\lambda$ is the uniform measure on the box $[0,L]^d$. 
Hence the number density, i.\,e.~the number of points per volume is given by $NL^{-d}$.
In other words, the volume per particle is $r^d$ where 
\begin{align}\label{fo02}
r:=L{N^{-\frac1d}},
\end{align}
so that the distance to the closest point is of order $r$.
A natural question is how the length scale (\ref{eq:matchingprob})
compares to the length scale (\ref{fo02}). Clearly, $R$ is at least of the order
of $r$. Interestingly, it turns out that despite the optimization, $R$ can be much larger than $r$,
at least in low dimensions.
%
The following theorem summarizes how the asymptotic matching differs from  the microscopic scale.

\begin{theorem}\label{thm:AKTasymptotic}
Let $X_1, \dots, X_N, Y_1, \dots, Y_N$ be i.\,i.\,d. random variables taking values in the box $[0,L]^d$ with uniform distribution, i.\,e.
\[
\PP(X_i \in A) = \frac{|A|}{L^d}, \quad \text{for every $A\subseteq [0,L]^d$ Borel}, 
\]
and $|A|$ denotes the Lebesgue measure of $A$. The following holds true
\begin{equation}\label{eq:thmasymptotic}
\sqrt{\EE \bigg[\frac1N \min_{\sigma \in \mathcal{S}_N}\sum_{n=1}^N |Y_{\sigma(n)} - X_n|^2\bigg]} \sim
r  \begin{cases}
\sqrt N & \text{if $d=1$},\\
\sqrt{ \ln N} & \text{if $d=2$}, \\
1 & \text{if $d>2$},
\end{cases}
\end{equation}
\end{theorem}

Here and in the sequel, we use the notation $A\lesssim B$ if there exists a finite constant $C$ only depending
on $d$, such that $A \le C B$. We write $A \sim B$ if both $A \lesssim B$ and $B \lesssim A$.
The critical dimension $2$ was firstly understood in the seminal work \cite{AKT84} and 
recently improved in \cite{AmStTr16} where the convergence of the rescaled cost was shown 
on the basis of a PDE ansatz made in \cite{CaLuPaSi14}. See also the monograph 
\cite[Chapter 4, Chapter 14, Chapter 15]{Ta14} where many aspects of the optimal matching problem are studied. 
In the special case of $d=1$ the left hand side of \eqref{eq:thmasymptotic} can be computed explicitly, i.\,e.
\[
\EE \bigg[\frac1N \min_{\sigma \in \mathcal{S}_N}\sum_{n=1}^N |Y_{\sigma(n)} - X_n|^2\bigg] = \frac1{3(N+1)},
\]
see the monograph \cite{BoLe19} for a detailed study of the optimal matching problem in dimension $1$.

\medskip
In establishing Theorem \ref{thm:AKTasymptotic}, we will rely
on the following elementary properties of the binomial distribution. 

\begin{lemma}\label{lem:lemma1}
Let $X_1, \dots, X_N$ be independent and uniformly distributed in $\Omega \subseteq \mathbb{R}^d$ with $|\Omega| < \infty$. For $Q \subset \Omega$ define the number of points that belongs to $Q$, i.\,e. 
\[
N_Q := \#\{ n \in \{0,\dots,N\} \ |\ X_n \in Q\},
\]
and let $\theta$ denote the volume fraction
\begin{equation}\label{eq:volfract}
\theta := \frac{|Q|}{|\Omega|} \in [0,1].
\end{equation}
Then the following properties hold
\begin{enumerate}[label=(\alph*)]
\item\label{eq:lemcomb} mean $\EE [N_Q] = N \theta$, \\
variance $\EE[(N_Q - \EE [N_Q])^2] = N \theta(1-\theta)$,\\
kurtosis $\EE [(N_Q - \EE[N_Q])^4] \le 3(N\theta(1-\theta))^2 + N \theta (1-\theta)$.

\item\label{eq:lemstoc} Concentration: provided $N\theta \ge 1$ \\
\[
(\EE[|\rho -1|^p])^\frac1p \lesssim \frac1{\sqrt{N\theta}} \quad\text{for $p \in \{2,4\}$}\footnote{with minor changes in the proof it can be shown that the statement holds for any $1<p<4$.} ,
\]
where $\rho:= \frac{N_Q}{\EE[N_Q]} = \frac{N_Q}{N\theta}$ so that $|\rho - 1|$ is the relative size of fluctuation. \\
\noindent
Stochastic $L^p$-norm bound:
\[
\bigg(\EE \bigg[ \bigg( I(N_Q \neq 0) \frac1{N_Q}\bigg)^2\bigg]\bigg)^\frac12 \lesssim \frac1{N\theta}.
\]
\end{enumerate}
\end{lemma}
\begin{proof}
{\sc Proof of }\ref{eq:lemcomb}. By definition of $N_Q$ we may write
\begin{equation}\label{eq:lemNQexpl}
N_Q = \sum_{n=1}^N I(X_n \in Q) = \sum_{n=1}^N I_n,
\end{equation}
where we set $I_n : = I (X_n \in Q)$. Note that since the $X_n,\, n \in \{0,\dots,N\}$, are independent and uniformly distributed in $\Omega$, the random variables $I_n$ are independent with mean given by
\begin{equation}\label{eq:Intriangl}
\EE [I_n] = \P(X_n \in Q) \stackrel{\eqref{eq:volfract}}{=} \frac{|Q|}{|\Omega|} = \theta, 
\end{equation}
in particular, 
$
\EE[N_Q] \stackrel{\eqref{eq:lemNQexpl}}{=} \sum_{n=1}^N \EE[I_n] = N \theta.
$
Note that
\begin{equation}\label{eq:lemma1star}
N_Q - \EE[N_Q] = \sum_{n=1}^N (I_n - \theta),
\end{equation}
therefore taking the square we obtain
\begin{equation}\label{eq:lemvarianceNQpath}
(N_Q - \EE[N_Q])^2 = \sum_{n=1}^N \sum_{m=1}^N (I_n - \theta)(I_m - \theta).
\end{equation}
Thus, taking the expectation in \eqref{eq:lemvarianceNQpath} yields
\begin{equation}\label{eq:lemma1stella2}
\begin{split}
\EE[ (N_Q - \EE[N_Q])^2 ] &= \sum_{n=1}^N \sum_{m=1}^N \EE [(I_n - \theta)(I_m - \theta)].
\end{split}
\end{equation}
Note that the expectation on the right hand side of \eqref{eq:lemma1stella2} vanishes whenever $n\neq m$. Indeed, by independence
\[
\EE [(I_n - \theta)(I_m - \theta)] = \EE [(I_n - \theta)]\EE[(I_m - \theta)] \stackrel{\eqref{eq:Intriangl}}{=} 0.
\]
Hence we may write
\begin{equation}\label{eq:lemma1stella}
\begin{split}
\EE[ (N_Q - \EE[N_Q])^2 ] &=\sum_{n=1}^N \EE [(I_n - \theta)^2] \\
& = N \EE[(I_1 - \theta)^2]\\
& = N \EE[I_1] - 2 \theta \EE[I_1] + \theta^2 \stackrel{\eqref{eq:volfract}}{=} \theta (1-\theta).
\end{split}
\end{equation}

\medskip
Let us now turn to last item in \ref{eq:lemcomb}. Taking the fourth power, by \eqref{eq:lemNQexpl} we may write 
\[
\begin{split}
\EE [ (N_Q - \EE [N_Q])^4] & \stackrel{\eqref{eq:lemma1star}}{=} \sum_{n_1=1}^N \sum_{n_2=1}^N \sum_{n_3=1}^N \sum_{n_4=1}^N \EE[ (I_{n_1} - \theta) (I_{n_2} - \theta)  (I_{n_3} - \theta)  (I_{n_4} - \theta) ] 
\end{split}.
\]
Note that as in \eqref{eq:lemma1stella2} the right hand side is not null if the indexes are all equal or pairwise equal.  Indeed, arguing by independence as before all of the other cases vanish. Hence we may write
\[
\begin{split}
\EE [ (N_Q - \EE [N_Q])^4] & \stackrel{\eqref{eq:lemma1star}}{=} 3 \sum_{n_1 =1}^N \sum_{\substack{n_2=1\\ n_2 \neq n_1}}^N \EE[ (I_{n_1} - \theta)^2]\EE[(I_{n_2} - \theta)^2 ]+ \sum_{n=1}^N \mathbb{E}[(I_n - \theta)^4] \\
& = 3N(N-1)(\EE[(I_1 - \theta)^2])^2 + N \EE[(I_1 - \theta)^4]\\
& \le 3(N\EE[(I_1-\theta)^2])^2 + N \EE[(I_1 - \theta)^2]\\
& \stackrel{\eqref{eq:lemma1stella}}{=} 3(N\theta(1-\theta))^2 + N \theta(1-\theta),
\end{split}
\]
where in the inequality we used the fact that $|I_1 - \theta|\le 1$.

\medskip
{\sc Proof of }\ref{eq:lemstoc}. By Jensen's inequality $\EE [(\rho-1)^2] \le (\EE[(\rho-1)^4])^\frac12$; thus it is sufficient to show the case $p=4$. By definition of $\rho$ we may write
\[
\begin{split}
\EE[(\rho-1)^4] & = \frac1{(\EE[N_Q])^4}\EE[(N_Q - \EE[N_Q])^4] \\
&\stackrel{\ref{eq:lemcomb}}{\le} \frac1{(N\theta)^4}(3(N\theta(1-\theta))^2+N\theta(1-\theta)) \\
& \le 3\bigg( \frac{1-\theta}{N\theta} \bigg)^2 + \frac{1-\theta}{(N\theta)^3} \\
& \lesssim \frac1{(N\theta)^2},
\end{split}
\] 
where the last inequality holds since $N\theta \ge 1 \ge 1-\theta$. 

\medskip
Let us now turn to the last item of \ref{eq:lemstoc}. By definition of $N_Q$ we may write
\[
\begin{split}
\EE\bigg[\bigg(I(N_Q\neq 0)\frac1{N_Q}\bigg)^2\bigg] & = \EE\bigg[I(N_Q\ge 1)\frac1{N_Q^2}\bigg] \\
& = \EE\bigg[I(1\le N_Q\le \frac12\EE[N_Q])\frac1{N_Q^2}+ I(N_Q > \frac12\EE[N_Q])\frac1{N_Q^2}\bigg]\\
& \le \PP\big( N_Q \le \frac12 \EE[N_Q]\big) + \frac4{(\EE[N_Q])^2}.
\end{split}
\]
By applying Chebyshev's inequality in the form of
\[
\PP\big( N_Q \le \frac12 \EE[N_Q]\big) \le \frac{16}{(\EE[N_Q])^4}\EE[(N_Q - \EE[N_Q])^4],
\]
we may write
\[
\begin{split}
\EE\bigg[\bigg(I(N_Q\neq 0)\frac1{N_Q}\bigg)^2\bigg] & \le \frac{16}{(\EE[N_Q])^4}\EE[(N_Q - \EE[N_Q])^4]  + \frac4{(\EE[N_Q])^2} \\
& \stackrel{\ref{eq:lemcomb}}{\le} \frac{16}{(N\theta)^4}(3(N\theta(1-\theta))^2+N\theta(1-\theta))  + \frac4{(\EE[N_Q])^2} \\
& \lesssim \frac1{(N\theta)^2},
\end{split}
\]
where the last inequality holds since $N\theta \ge 1\ge 1-\theta$.
\end{proof}

\section{Upper Bound}
In this section we show the upper bound part of Theorem \ref{thm:AKTasymptotic}.
Our intermediate goal is to consider a semidiscrete version of \eqref{eq:matchingprob}, i.\,e. to compare the empirical measure $\frac1N \sum_{n=1}^N \delta_{X_n}$ and the uniform measure $\lambda = \frac1{L^d}\,\mathrm{d}x\mres [0,L]^d$. Given a point cloud $\{X_n\}_{n=1}^N$ we will construct (thanks to an algorithm) a map 
\[
T: [0,L]^d \rightarrow \{ X_n\}_{n=1}^N \quad \text{with} \quad |T^{-1}(X_n)| = \frac{L^d}N,
\]
which is close to the identity in a square averaged sense, that is
\begin{equation}\label{eq:uppmapT}
\int_{[0,L]^d} |T - \id|^2 \, \mathrm{d} \lambda = \frac 1 {L^d} \int_{[0,L]^d} |T(x)-x|^2 \, \mathrm{d} x \quad \text{is small.}
\end{equation}

\begin{figure}[ht]
\center
\includegraphics[scale=0.75]{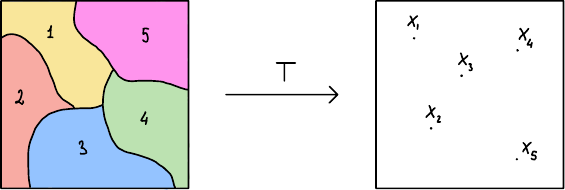}
\caption{The map $T$.}
\label{fig:mapT}
\end{figure}
\noindent
Note that $\{T^{-1}(X_n)\}_{n=1}^N$ amounts to an equipartition of $[0,L]^d$ and $T$ associates each cell $T^{-1}(X_n)$ to a point $X_n$ as in Figure \ref{fig:mapT}. Note that in general $X_n$ does not belong to $T^{-1}(X_n)$ for $N \gg 1$. In  mathematical language we say that $T_\# \lambda$ is the \emph{push forward} (in symbols $T_\# \lambda = \frac1N \sum_{n=1}^N \delta_{X_n}$) of the uniform measure to the empirical measure, which amounts to
\begin{equation}\label{eq:def_pushforward}
\lambda(T^{-1}(A)) = \bigg(\frac 1N \sum_{n=1}^N \delta_{X_n}\bigg)(A) \quad \text{for every Borel set $A \subseteq [0,L]^d$.}
\end{equation}

\medskip
Moreover, as shown in the next proposition, we can construct the map $T$ in such a way that the quantity \eqref{eq:uppmapT} is bounded from above by the right hand side of \eqref{eq:thmasymptotic}.
%
\begin{proposition}\label{prop:mapupper}
Let $X_1, \dots, X_N$ and $r$ be as in Theorem \ref{thm:AKTasymptotic}. There exists a random map $T:[0,L]^d \rightarrow [0,L]^d$  such that its preimage satisfies $|T^{-1}(X_i)| = L^d N^{-1}$ and its $L^2$ norm satisfies the upper bound
\begin{equation}\label{eq:propmapbound}
\bigg(\EE \bigg[ \frac1{L^d} \int_{[0,L]^d} |T - \id|^2 \bigg]\bigg)^\frac12 \lesssim 
r  \begin{cases}
\sqrt{N} & \text{if $d=1$},\\
\sqrt{\ln N} & \text{if $d=2$}, \\
1 & \text{if $d>2$}.
\end{cases}
\end{equation}
\end{proposition}
\begin{proof}[Proof of the upper bound of Theorem \ref{thm:AKTasymptotic}]
{\sc Step 1.} Kantorovich's relaxation. We argue that the following equality holds
\begin{equation}\label{eq:MontoKant}
\begin{split}
\min_{\sigma \in \mathcal{S}_N} \sum_{n=1}^N |Y_{\sigma(n)} - X_n|^2 = \min_{\pi \in B_N} \frac1N \sum_{n=1}^N \sum_{m=1}^N |Y_m - X_n|^2 \pi_{nm}
\end{split}
\end{equation}
where $B_N$ denotes the set of bistochastic\footnote{also called doubly stochastic} matrices (sometimes called Birkhoff polytope), namely the set
\begin{equation}\label{eq:Birkhoffpolytope}
B_N = \big\{ \pi := (\pi_{nm})_{n,m=1}^N \in \mathbb{R}^{N\times N} \ | \ \pi_{nm}\ge 0, \sum_{n=1}^N \pi_{nm} =1, \sum_{m=1}^N \pi_{nm} =1 \big\}.
\end{equation}
Indeed, the map
\[
(\pi_{n,m})_{n,m=1}^N \mapsto \frac1N \sum_{n=1}^N\sum_{m=1}^N |Y_m - X_n|^2 \pi_{nm},
\]
is linear and the set $B_N$ is convex. Hence the minimum on the right hand side of \eqref{eq:MontoKant} is attained at the extremal points of $B_N$. Finally, the Birkhoff-von Neumann Theorem states that the extremal points of $B_N$ correspond to permutation matrices, i.\,e. $\pi_{nm} = \delta_{\sigma(n)m}$ for some $\sigma \in S_N$, where $\delta_{ij}$ denotes the Kronecker delta.  To be self contained, we briefly reproduce an easy argument for the Birkhoff-von Neumann Theorem that can be found in \cite{Hurlbert2012ASP}. Let $Q$ be an extremal point of $B_N$, we claim that $Q$ is a permutation matrix. By the constraints in \eqref{eq:Birkhoffpolytope} it is enough to show that $Q$ has integer entries (in particular they are $0$ or $1$). We claim that if the extremal point $Q$ has non integer entries it is the midpoint of a segment line contained in $B_N$. Let us suppose that there exists a couple of indexes $(n_1,m_1)$ such that $0<Q_{n_1, m_1}<1$. By the column constraint $\sum_{n=1}^N Q_{n,m} = 1$ there exists an index $n_2$ such that $0<Q_{n_2, m_1}<1$, likewise because of the row constraint $\sum_{m=1}^N Q_{n,m} = 1$ there exists an index $m_2$ such that $0<Q_{n_2, m_2}<1$. We iterate this process and we stop when a couple of indexes $(n_i,m_i)$ for some $i$ is repeated. Moreover, we may consider the shortest sequence of such indexes so that the final index is $(n_1,m_1)$, see Figure \ref{fig:BirkVonFig}. It is not hard to check that the resulting  sequence $\{(n_i,m_j) \, : \, j-i \le 1\}_{i,j=1}^N$ is even, namely there exist an index $k$ such that $(n_k, m_k) = (n_1, m_1)$.
Indeed, if that is not the case the indexes $(n_{k+1}, m_k)$, $(n_1, m_1)$ and $(n_2, m_1)$ belong to the same column. By deleting $(n_2, s_1)$ and starting at $(n_2, m_2)$ we obtain an admissible shorter sequence.
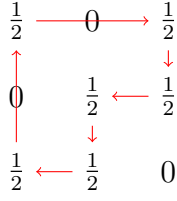
\begin{figure}[ht]
\begin{tikzpicture}[node distance=2cm]
\node (A) at (0, 0) {$\frac12$};
\node (B) at (0, 1) {0};
\node (C) at (0, 2) {$\frac12$};
\node (D) at (1, 0) {$\frac12$};
\node (E) at (1, 1) {$\frac12$};
\node (F) at (1, 2) {0};
\node (G) at (2, 0) {0};
\node (H) at (2, 1) {$\frac12$};
\node (I) at (2, 2) {$\frac12$};

\draw[red, ->]
  (H) edge (E) (E) edge (D) (D) edge (A) (A) edge (C) (C) edge (I) (I) edge (H);

\end{tikzpicture}
\caption{A possible sequence $\{n_i,m_j\}$.}
\label{fig:BirkVonFig}
\end{figure}
Given such a sequence let $\varepsilon_0 = \min \{ Q_{n_i,m_j} \} >0$. For any $0<\varepsilon<\varepsilon_0$ define $Q^+(\varepsilon)$ (respectively $Q^-(\varepsilon)$) by increasing (respectively decreasing) the value of each $Q_{n_i,m_i}$ by $\varepsilon$, while decreasing (respectively increasing) the value of each $Q_{n_i,m_{i+1}}$ by $\varepsilon$. Finally, by definition $Q^+(\varepsilon)$ and $Q^-(\varepsilon)$ belong to $B_N$, the line segment connecting $Q^+(\varepsilon)$ and $Q^-(\varepsilon)$ is contained in $B_N$ and has $Q$ as its midpoint.

\medskip
{\sc Step 2.} Triangle inequality. Let $T,S:[0,L]^d \rightarrow [0,L]^d$ be maps that satisfy \eqref{eq:propmapbound} for $(X_n)_{n=1}^N$ and $(Y_m)_{m=1}^N$ respectively. We claim that 
\begin{equation}\label{eq:thmtriangleupp}
\begin{split}
\lefteqn{\bigg( \EE\bigg[ \min_{\pi \in B_N} \frac1N \sum_{n=1}^N \sum_{m=1}^N |Y_m - X_n|^2 \pi_{nm}\bigg]\bigg)^\frac12} \\
& \le \bigg( \EE\bigg[\frac1{L^d}\int_{[0,L]^d}|S- \id|^2\bigg]\bigg)^\frac12 + \bigg(\EE\bigg[\frac1{L^d}\int_{[0,L]^d}|T - \id|^2\bigg]\bigg)^\frac12.
\end{split}
\end{equation}
Indeed, define the matrix $\bar{\pi} = (\bar{\pi}_{nm})_{n,m=1}^N$ with entries given by
\begin{equation}\label{eq:thmmatrixdef}
\bar{\pi}_{nm} := \frac N{L^d}|T^{-1}(X_n) \cap S^{-1}(Y_m)|.
\end{equation}
By construction $\bar{\pi}$ is bistochastic. Indeed, $(T^{-1}(X_n))_{n=1}^N$ and $(S^{-1}(Y_m))_{m=1}^N$ provide two partitions of $[0,L]^d$ and each cell has total mass given by ${L^d}N^{-1}$. Therefore for every column $m$ we have $\sum_{n=1}^N \bar{\pi}_{n,m} = N L^{-d} |S^{-1}(Y_m)|=1$, and similarly for the rows. Furthermore, 
\begin{equation}\label{eq:candtriangleupp}
\bigg( \frac1N \sum_{n=1}^N \sum_{m=1}^N |Y_m - X_n|^2 \bar{\pi}_{nm}\bigg)^\frac12  \le \bigg( \frac1{L^d}\int_{[0,L]^d}|S- \id|^2\bigg)^\frac12 + \bigg(\frac1{L^d}\int_{[0,L]^d}|T - \id|^2\bigg)^\frac12.
\end{equation}
Indeed, by the triangle inequality
\[
\begin{split}
\lefteqn{
\bigg( \frac1N \sum_{n=1}^N \sum_{m=1}^N |Y_m - X_n|^2 \bar{\pi}_{nm}\bigg)^\frac12
}\\
& \stackrel{\eqref{eq:thmmatrixdef}}{=} \bigg( \frac1{L^d} \sum_{n=1}^N \sum_{m=1}^N |Y_m - X_n|^2 |T^{-1}(X_n) \cap S^{-1}(Y_m)|\bigg)^\frac12 \\
& = \bigg( \frac1{L^d} \sum_{n=1}^N \sum_{m=1}^N \int_{T^{-1}(X_n)\cap S^{-1}(Y_m)} |(Y_m - \id) - (X_n - \id)|^2 \bigg)^\frac12 \\
& \le \bigg( \frac1{L^d} \sum_{n=1}^N \sum_{m=1}^N \int_{T^{-1}(X_n)\cap S^{-1}(Y_m)} |Y_m - \id|^2\bigg)^\frac12 \\
& + \bigg( \frac1{L^d} \sum_{n=1}^N \sum_{m=1}^N \int_{T^{-1}(X_n)\cap S^{-1}(Y_m)} |X_n - \id|^2 \bigg)^\frac12\\
& = \bigg(\sum_{m=1}^N \int_{S^{-1}(Y_m)}|Y_m - \id|^2\bigg)^\frac12 + \bigg(\sum_{n=1}^N \int_{T^{-1}(X_n)}|X_n - \id|^2\bigg)^\frac12\\
& =\bigg( \frac1{L^d}\int_{[0,L]^d}|S- \id|^2\bigg)^\frac12 + \bigg(\frac1{L^d}\int_{[0,L]^d}|T - \id|^2\bigg)^\frac12.
\end{split}
\]
Taking the expectation in \eqref{eq:candtriangleupp} and the triangle inequality w.\,r.\,t. to $(\EE[|\cdot|])^\frac12$ implies \eqref{eq:thmtriangleupp}. 

\medskip
{\sc Step 3.} Conclusion. Finally, \eqref{eq:MontoKant}, \eqref{eq:thmtriangleupp} and \eqref{eq:propmapbound} combine to the upper bound in \eqref{eq:thmasymptotic}. 
\end{proof}

We now turn to the proof of Proposition \ref{prop:mapupper}. We start by introducing the building block for the map $T$ that attains the desired upper bound, see Figure \ref{fig:auxmap1d}. 
\begin{lemma}\label{lem:auxmap1d}
Let $\rho_- \in [0,2]$, define the measure
\begin{equation}\label{eq:defrhomeas}
\dd \rho = \rho_-\dd x\mres[0,1] + (2-\rho_-)\dd x\mres[1,2],
\end{equation}
and let $T_{\rho_-}:[0,2] \rightarrow [0,2]$ be the map
\[
T_{\rho_-}(x) :=\footnote{with the understanding that 
\[
T_0 (x) := 
\begin{cases}
1 & \text{for $x = 0$}, \\
1 + \frac x2 & \text{for $x > 0$},
\end{cases}
\quad \text{and} \quad
T_2 := \frac x 2 \quad \text{for $x \le 2$}.
\]
}
\begin{cases}
\frac x {\rho_-} & \text{for $x\le \rho_-$}, \\
2 - \frac{2-x}{2-\rho_-} & \text{for $x > \rho_-$}.
\end{cases}
\]
Then $\rho$ is the push-forward of the Lebesgue measure under $T_{\rho_-}$, in symbols $\dd \rho = (T_{\rho_-})_\# \dd x$, and the map $T_{\rho_-}$ satisfies the estimates
\begin{equation}\label{eq:auxmap1}
\int_{[0,2]} (T_{\rho_-} - \id)^2 \lesssim (\rho_- - 1)^2,
\end{equation}
and
\begin{equation}\label{eq:auxmap2}
\int_{[0,2]} \bigg(\frac12(T_{\rho_-} - \id)+\frac12(T_{2-\rho_-} - \id)\bigg)^2 \lesssim (\rho_- - 1)^4.
\end{equation}
\end{lemma}
%
%
%
%
%
\begin{proof}
Note that by definition both of the items in \eqref{eq:auxmap1} and \eqref{eq:auxmap2} are $\lesssim 1$, thus w.\,l.\,o.\,g. we may assume that $|\rho_- - 1| \ll 1$. Hence, 
\[
\frac1{\rho_-} = \frac1{1+(\rho_- - 1)} = 1 - (\rho_- - 1) + O((\rho_- - 1)^2) = 2 - \rho_- + O((\rho_- - 1)^2)
\]
and
\[
\frac1{2-\rho_-} = \frac1{1-(\rho_- - 1)} = 1 + (\rho_- - 1) + O ((\rho_- - 1)^2) = \rho_- + O((\rho_- - 1)^2). 
\]
By definition of $T_{\rho_-}$ the last two imply
\[
T_{\rho_-}(x) = 
\left\{\begin{array}{lr}
(2-\rho_-) x & \text{for $x \le 1$}\\
2 - \rho_-(2-x) & \text{for $x >1$}
\end{array}\right\}
+ O((\rho_- - 1)^2).
\]
\begin{figure}[ht]
\center
\includegraphics[scale=0.75]{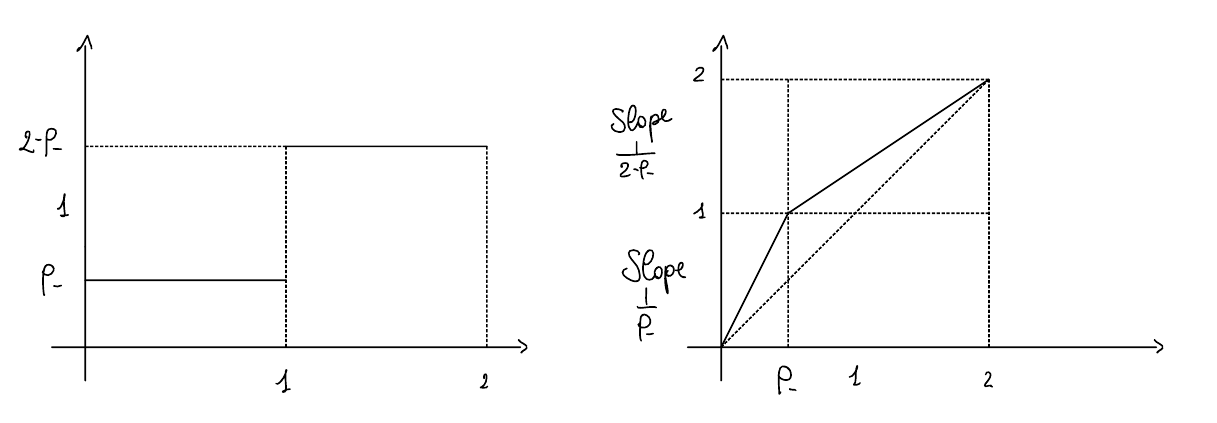}
\caption{Graph of $\rho_-$ and $T_{\rho_-}$.}
\label{fig:auxmap1d}
\end{figure}
For the displacement we thus obtain
\[
T_{\rho_-}(x) - x = 
(1-\rho_-)
\left\{\begin{array}{lr}
 x & \text{for $x \le 1$}\\
2-x & \text{for $x >1$}
\end{array}\right\}
+ O((\rho_- - 1)^2).
\]
Therefore, taking the square
\[
(T_{\rho_-}(x) - x)^2 = O((\rho_- - 1)^2)
\]
and \eqref{eq:auxmap1} is established. Finally, \eqref{eq:auxmap2} follows easily by the observation $1-(2-\rho_-) = -(1-\rho_-)$ which implies
\[
T_{2-\rho_-}(x) - x = - (T_{\rho_-} (x) - x) + O((\rho_- - 1)^2).
\]
\end{proof}
We are now in a position to prove Proposition \ref{prop:mapupper}.
\begin{proof}[Proof of Proposition \ref{prop:mapupper}]
{\sc Step 1.} The construction of $T$. \textbf{The boxes $Q$.} Consider the family of nested boxes $\{Q\}$ defined by the dyadic decomposition as in Figure \ref{fig:dyadicdecompos}.
\begin{figure}[ht]
\center
\includegraphics[scale=0.75]{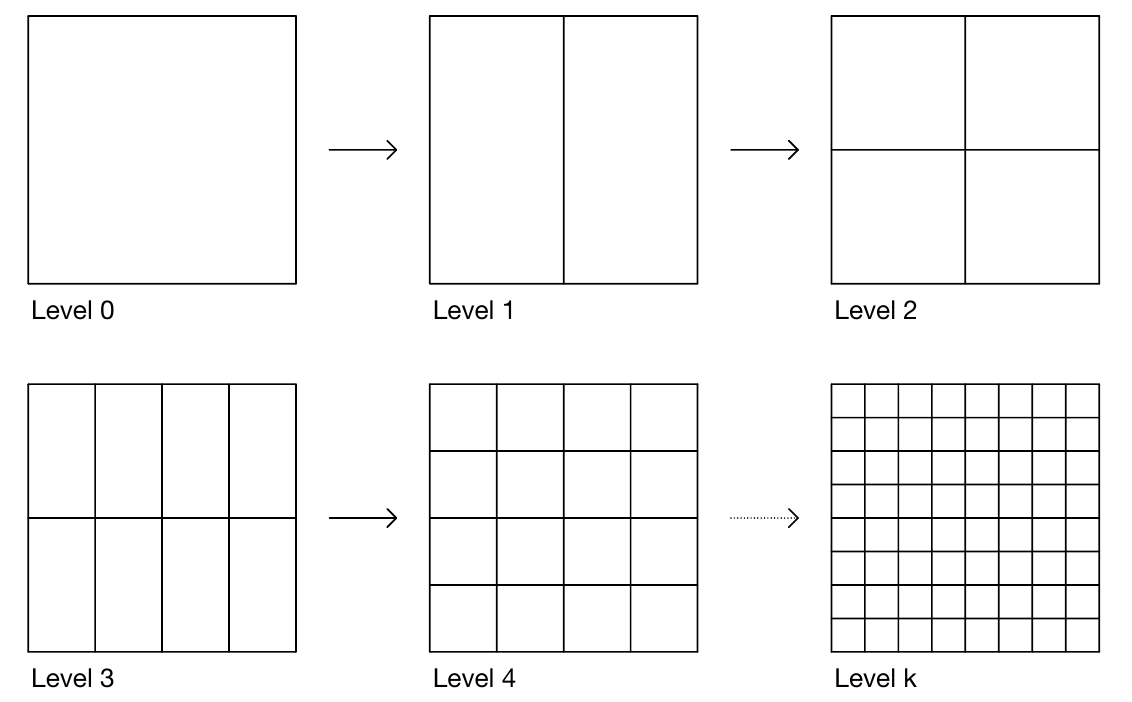}
\caption{Dyadic decomposition.}
\label{fig:dyadicdecompos}
\end{figure}
Clearly by construction the number of boxes on the $k$-th level is equal to $2^k$ and for every $Q$ on level $k$ we have that $|Q| = 2^{-k}L^d$. Moreover, letting $L_k$ denoting the side length of a box $Q$ on level $k$, it holds that 
\begin{equation}\label{eq:conLk}
L_k = 2^{- \lceil\frac k d \rceil} L,
\end{equation}
where $\lceil \cdot \rceil$ denotes the ceiling function. Let $N_Q$ be as in Lemma \ref{lem:lemma1}. We want to stop the decomposition when $N_Q = O(1)$, which amounts to typically having an $O(1)$ number of points. Recalling that $r$ denotes the typical distance to the closest point, i.\,e. $r =  LN^{-\frac1d}$, we stop at the scale $L_{k_*}$ such that 
\begin{equation}\label{eq:stoppingscale}
2r > L_{k_*} \ge r.
\end{equation}

\medskip
\textbf{The number densities $\rho_k$.} We consider the piecewise constant number densities
\begin{equation}\label{eq:defrhok}
\rho_k(x) := \frac{N_Q}{L^d \EE[N_Q]} \stackrel{\ref{eq:lemcomb}}{=} \frac{N_Q}{N|Q|} \quad \text{for $x \in Q$}.
\end{equation}
\begin{figure}[ht]
\center
\includegraphics[scale=0.75]{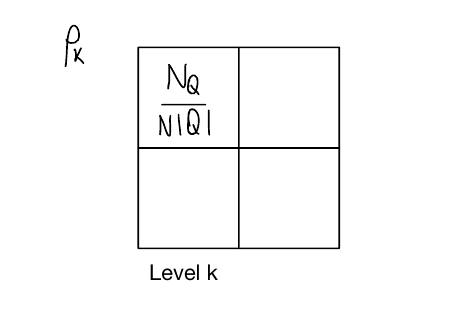}
\caption{Density $\rho_k$.}
\end{figure}
Note that the $\rho_k$ are normalized, indeed
\[
\rho_k ([0,L]^d) = \sum_{\text{$Q$ of level $k$}} |Q| \frac{N_Q}{N|Q|} = \frac1N \sum_{\text{$Q$ of level $k$}} N_Q = 1.
\]

\medskip
\textbf{The maps $T_k$.} 
Let $Q$ be a box of level $k$. By translational invariance we may assume w.\,l.\,o.\,g. that $Q$ is the box that contains the origin. By rotational invariance we may assume that on level $k+1$, $Q$ is divided along the first coordinate direction, that is into a left part $Q_-$ and a right part. Let 
\begin{equation}\label{eq:numbdensleft}
\rho_- := \frac{2 N_{Q_-}}{N_Q}\in [0,2]
\end{equation}
be its number density and let $T_{\rho_-}$ be the map of Lemma \ref{lem:auxmap1d}.
Define 
\begin{equation}\label{eq:mapTk}
T_k (x) = \bigg(L_k T_{\rho_-}\bigg( \frac{x_1}{L_k} \bigg),x_2,\dots, x_d \bigg) \quad \text{for $x\in Q$},
\end{equation} 
let $\rho_0 = \frac1{L^d}\dd x \mres [0,L]^d$ and note that by construction 
\begin{equation}\label{eq:pushforwrhok}
\rho_k = (T_k)_\# \rho_{k-1}
\end{equation} 
is the pushforward of $\rho_{k-1}$ under $T_k$. 
\begin{figure}[ht]
\center
\includegraphics[scale=0.75]{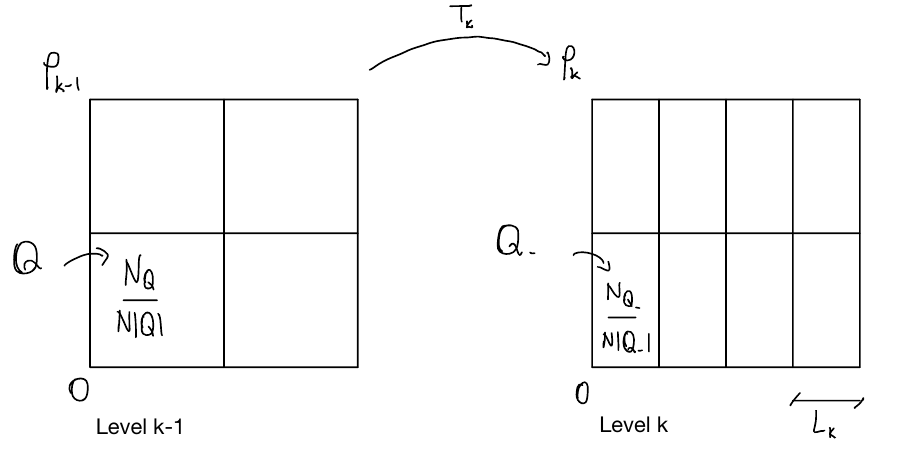}
\caption{The map $T_k$.}
\end{figure}
\begin{figure}[ht]
\center
\includegraphics[scale=0.75]{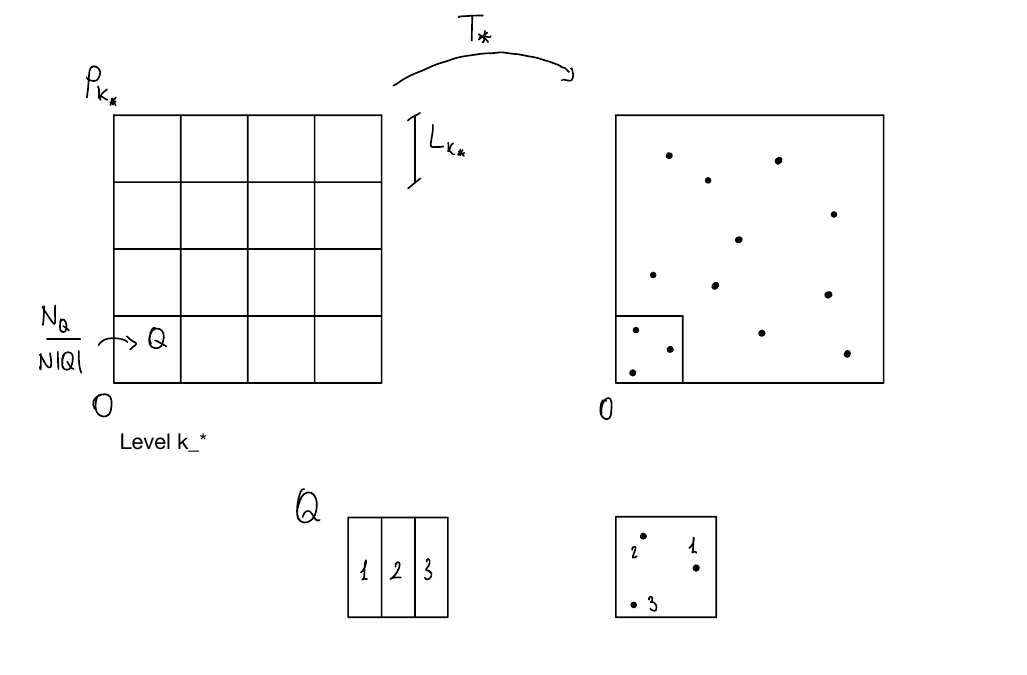}
\caption{Construction for the last mile.}
\end{figure}

\medskip
On the \emph{``stopping''} level $k_*$ we set $T_*$ such that if on a box of level $k_*$ there are no points $T_* = \id$, otherwise if on $Q$ there are $n\ge1$ points, we divide $Q$ into $n$ cells of equal volume, we enumerate the points inside the box and send the $n$-th point to the $n$-th cell of the box. Note that, by the aforementioned construction of $T_*$ we have
\begin{equation}\label{eq:pushforwstar}
(T_*)_\# \rho_{k_*} = \frac1N \sum_{i=1}^N \delta_{X_n}
\end{equation}
and
\begin{equation}\label{eq:stopscalTstar}
|T_*(x) - x| \le \sqrt{d} L_{k_*} \sim r.
\end{equation}
We then set 
\[
T:= T_* \circ T_{k_*} \circ \dots \circ T_1.
\]
Note that by construction, \eqref{eq:pushforwrhok} and \eqref{eq:pushforwstar} we have
\[
\frac1N \sum_{n=1}^N \delta_{X_n} = T_\#\dd \rho_0 = T_\# \bigg(\frac1{L^d} \dd x \mres [0,L]^d\bigg)
\]
and thus by the definition of push forward \eqref{eq:def_pushforward}
\[
|T^{-1}(X_n)| = L^d N^{-1}.
\]

\medskip
{\sc Step 2.} The estimates on $T$. Let $S_0 = \id$ and let
\[
S_k := T_k \circ \dots \circ T_1
\]
so that $S_k = T_k \circ S_{k-1}$. We claim that there exists a constant $C$ such that 
\begin{equation}\label{eq:propproofbase}
\EE \bigg[ \int |S_k - \id|^2 \dd \rho_0 \bigg] \le C \bigg( \frac{L_k}r\bigg)^{2-d} r^2 + \bigg( 1 + C \bigg(\frac r{L_k}\bigg)^d \bigg) \EE \bigg[ \int |S_{k-1} - \id|^2 \dd\rho_0 \bigg].
\end{equation}
Note that by the definition of $S_k$, we may write
\[
S_k - \id = (T_k - \id)\circ S_{k-1} + (S_{k-1} - \id),
\]
so that taking the square
\[
|S_k - \id|^2 = |S_{k-1} - \id|^2 + |(T_k - \id)\circ S_{k-1}|^2 + 2(S_{k-1} - \id)\cdot (T_k - \id) \circ S_{k-1}. 
\]
Finally, integrating the latter against $\dd \rho_0$ and taking the expectation will lead to \eqref{eq:propproofbase} once we can estimate the right hand side of 
\[
\begin{split}
\EE\bigg[\int |S_k - \id|^2 \dd \rho_0 \bigg] & =
\EE\bigg[\int |S_{k-1} - \id|^2 \dd \rho_0\bigg] + \EE\bigg[\int |(T_k - \id)\circ S_{k-1}|^2 \dd \rho_0\bigg] \\
& + 2 \EE\bigg[\int (S_{k-1} - \id)\cdot (T_k - \id) \circ S_{k-1} \dd \rho_0 \bigg]. 
\end{split}
\]

\medskip
{\sc Step 2.1.} We claim that 
\begin{equation}\label{eq:propdeprev}
\EE\bigg[\int |(T_k - \id)\circ S_{k-1}|^2 \dd \rho_0\bigg] \lesssim \bigg( \frac{L_k}r\bigg)^{2-d} r^2.
\end{equation}
Indeed, note that 
\begin{equation}\label{eq:rhokSk}
\dd \rho_{k-1} = (S_{k-1})_\# \dd \rho_0;
\end{equation}
hence we may write
\[
\int|(T_k - \id)\circ S_{k-1}|^2 \dd \rho_0 = \int|T_k - \id|^2 \dd \rho_{k-1}.
\]
By definition \eqref{eq:defrhok} of $\rho_{k-1}$,
\[
\int|T_k - \id|^2 \dd \rho_{k-1} = \sum_{\text{$Q$ on level $k-1$}} \frac{N_Q}{N|Q|} \int_Q |T_k - \id|^2.
\]
By definition \eqref{eq:mapTk}, the map $T_k$ transports only along one direction acting as a rescaled version of the map $T_{\rho_-}$ of Lemma \ref{lem:lemma1}. Thus, by a change of variables,
\[
\frac1{|Q|} \int_Q |T_k - \id|^2 = \frac{L_k^2}2 \int_{[0,2]} (T_{\rho_-} - \id)^2 \quad \text{for $\rho_- = \frac{2 N_{Q_-}}{N_Q} \in[0,2]$.}
\]
By \eqref{eq:auxmap1}
\[
\int_{[0,2]} (T_{\rho_-} - \id)^2 \lesssim (\rho_- - 1)^2,
\]
so that
\[
\int |(T_k - \id)\circ S_{k-1}|^2 \dd \rho_0  \lesssim 
 \frac{L_k^2}N \sum_{\text{$Q$ on level $k-1$}} N_Q \bigg( \frac{2N_{Q_-}}{N_Q} - 1\bigg)^2.
\]
In order to compute the expectation it is convenient to introduce the expectation conditioned on the $\sigma$-algebra generated by $N_Q$ of level at most $k-1$, that is, 
\begin{equation}\label{eq:sigmaalg}
\mathcal{F}_k := \sigma(N_Q \ | \  \text{$Q$ of level $\le k-1$}), 
\end{equation}
which we characterize now. To this purpose, note that for any $n\in \{1,\dots,N\}$ and any $Q \subset [0,L]^d$, $X_n$ conditioned on the event $\{X_n \in Q\}$  is uniformly distributed on $Q$. Indeed, by the Bayes' theorem we have that for any $A \subset [0,L]^d$
\[
\mathbb{P}(X_n \in A \ | \ X_n \in Q ) = \frac{\mathbb{P}(\{X_n \in A\}\cap\{X_n \in Q\})}{\mathbb{P}(X_n\in Q)} = \frac{\mathbb{P}(X_n \in A \cap Q)}{\mathbb{P}(X_n \in Q)} = \frac{|A \cap Q|}{|Q|}.
\]
Hence, by independence of $X_1, \dots, X_N$, the law of $N_{Q_-}$ conditioned on $\mathcal{F}_{k-1}$ is equal to the law of $N_{Q_-}$ for $N_Q$ independent and uniformly distributed points on $Q$. Thus, by Lemma \ref{lem:lemma1} we have
\begin{equation}\label{eq:propsigmaalg}
\begin{split}
&\EE[N_{Q_-} \ | \ \mathcal{F}_{k-1}] = \frac12 N_Q \quad \text{since $\theta = \frac{|Q_-|}{|Q|}=\frac12$,}\\
&\EE[(N_{Q_-} - \frac12 N_Q)^2\ | \ \mathcal{F}_{k-1}] = \frac14 N_Q \quad \text{since $\theta(1-\theta) = \frac14$,} \\
&\EE\bigg[N_Q\bigg(\frac{2N_{Q_-}}{N_Q}-1\bigg)^2\bigg] = \EE\bigg[ N_Q \EE\bigg[\bigg(\frac{2N_{Q_-}}{N_Q}-1\bigg)^2 \ \bigg| \ \mathcal{F}_{k-1}\bigg]\bigg] = 1,
\end{split}
\end{equation}
where in the last item we used the fact that $N_Q$ is $\mathcal{F}_{k-1}$ measurable by construction of the latter. By the last identity we obtain from \eqref{eq:conLk}, the fact that $2^{k-1} \sim \big( \frac L{L_k} \big)^d$ and $r^d= \frac{L^d}N$ we get
\[
\EE\bigg[ \int |(T_k-\id) \circ S_{k-1}|^2\dd \rho_0 \bigg] \lesssim \frac{L_k^2}N \#\{\text{$Q$ on level $k-1$}\} = 2^{k-1}\frac{L_k^2}N  \sim r^2 \bigg( \frac{L_k}r\bigg)^{2-d}.
\]

\medskip
{\sc Step 2.2.} We claim that 
\[
\bigg|\EE \bigg[ \int (S_{k-1} - \id)\cdot (T_k - \id) \circ S_{k-1} \dd \rho_0 \bigg]  \bigg| \lesssim \bigg(\bigg( \frac{L_k}r\bigg)^{2-2d} r^2 \EE \bigg[ \int |S_{k-1} - \id|^2 \dd \rho_0 \bigg]\bigg)^\frac12,
\]
so that by Young's product inequality we may write
\[
\begin{split}
\bigg|\EE \bigg[ \int (S_{k-1} - \id)\cdot (T_k - \id) \circ S_{k-1} \dd \rho_0 \bigg]  \bigg|
 \lesssim \bigg(\frac{L_k}r\bigg)^{2-d}r^2 + \bigg(\frac r{L_k}\bigg)^d \EE \bigg[ \int |S_{k-1} - \id|^2 \dd \rho_0 \bigg].
\end{split}
\]
Note that $S_{k-1}$ depends on $\{X_1, \dots, X_N\}$ only through $N_Q$ of level $\le k-1$ so that $S_{k-1}$ is $\mathcal{F}_{k-1}$-measurable. Thus, by the Cauchy-Schwarz inequality with respect to $(\EE[\int|\cdot|^2\, \dd \rho_0])^\frac12$, \eqref{eq:rhokSk} and the tower property of the conditional expectation we may write
\[
\begin{split}
& \bigg|\EE \bigg[ \int (S_{k-1} - \id)\cdot (T_k - \id) \circ S_{k-1} \dd \rho_0 \bigg]  \bigg| \\
& = \bigg| \EE\bigg[\int (S_{k-1} - \id) \cdot \EE[T_k - \id \ | \ \mathcal{F}_{k-1}] \circ S_{k-1} \dd \rho_0 \bigg] \bigg| \\
& \le \bigg(\EE \bigg[ \int |S_{k-1} - \id|^2 \dd \rho_0 \bigg]\bigg)^\frac12 \bigg(\EE\bigg[ \int|\EE[T_k - \id \ | \ \mathcal{F}_{k-1}]\circ S_{k-1}|^2 \dd \rho_0 \bigg]\bigg)^\frac12\\
& = \bigg(\EE \bigg[ \int |S_{k-1} - \id|^2 \dd \rho_0 \bigg]\bigg)^\frac12 \bigg(\EE\bigg[ \int|\EE[T_k - \id \ | \ \mathcal{F}_{k-1}]|^2 \dd \rho_{k-1} \bigg]\bigg)^\frac12.
\end{split}
\]
Hence it is enough to estimate the last term on the r.\,h.\,s.. By definition \eqref{eq:defrhok} of $\dd \rho_{k-1}$ and by definition \eqref{eq:mapTk} $T_{k-1}$ we may write
\[
\begin{split}
\EE \bigg[ \int| \EE[T_k - \id \ | \ \mathcal{F}_{k-1}]|^2 \dd \rho_{k-1}\bigg] & = \EE\bigg[ \sum_{\text{$Q$ on level $k-1$}} \frac{N_Q}N \frac1{|Q|} \int_Q | \EE[T_{k} - \id \ | \ \mathcal{F}_{k-1}]|^2\bigg] \\ 
& = \EE\bigg[ \sum_{\text{$Q$ on level $k-1$}} \frac{N_Q}N L_k^2 \int_{[0,2]}(\EE[T_{\rho_-} - \id \ | \ \mathcal{F}_{k-1}])^2\bigg].
\end{split}
\]
We note that since $\rho_- = \frac{2N_{Q_-}}{N_Q}$, $2- \rho_-$ has the same law under $\EE[\ \cdot\ | \ \mathcal{F}_{k-1}]$. Indeed, consider the transformation of $[0,L]^d$ that swaps the right and left half of $Q$ as in Figure \ref{fig:swap}.
\begin{figure}[ht]
\center
\includegraphics[scale=0.75]{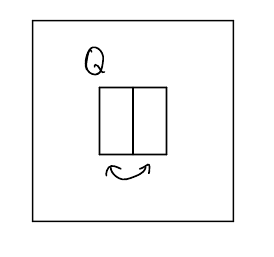}
\caption{Map that swaps the right and left side of $Q$.}
\label{fig:swap}
\end{figure}
Such a transformation preserves the law of the point cloud $\{X_1, \dots, X_N\}$, it preserves the observables $\{N_Q \ | \ \text{$Q$ on level $\le k-1$}\}$ and it converts $\rho_-$ into $2-\rho_-$. Due to this invariance in law,
\[
\EE[T_{\rho_-} - \id \ | \ \mathcal{F}_{k-1}] = \EE\bigg[\frac12(T_{\rho_-} - \id) + \frac12(T_{2-\rho_-} - \id) \ | \ \mathcal{F}_{k-1}\bigg],
\]
and thus by Jensen's inequality and Lemma \ref{lem:lemma1}
\[
\begin{split}
\int_{[0,2]}(\EE[T_{\rho_-} - \id \ | \ \mathcal{F}_{k-1}])^2 & \le \EE \bigg[ \int_{[0,2]} \bigg( \frac12(T_{\rho_-} - \id) + \frac12(T_{2 - \rho_-} - \id)  \bigg)^2 \ | \ \mathcal{F}_{k-1} \bigg] \\
& \stackrel{\eqref{eq:auxmap2}}{\lesssim} \EE[(\rho_- - 1)^4 \ | \ \mathcal{F}_{k-1}] \stackrel{\ref{eq:lemstoc}}{\lesssim} \frac1{N_Q^2}.
\end{split}
\]
Thus
\[
\begin{split}
\EE\bigg[ \int |\EE[T_k - \id \ | \ \mathcal{F}_{k-1} ]|^2 \dd \rho_{k-1}\bigg] & \lesssim \frac{L^2_k}N \sum_{\text{$Q$ on level $k-1$}} \EE\bigg[N_Q \frac1{N_Q^2}\bigg] \\
& = \frac{L^2_k}N \sum_{\text{$Q$ on level $k-1$}} \EE\bigg[I(N_Q \neq 0) \frac1{N_Q}\bigg] \\
& \stackrel{\ref{eq:lemstoc}}{\lesssim} \frac{L^2_k}{N^2} (2^{k-1})^2 \sim \bigg(\frac{L^d}N \bigg)^2 L_k^{2-2d} = \bigg( \frac{L_k}r \bigg)^{2-2d} r^2 .
\end{split}
\]

{\sc Step 3.} Conclusion. We finally show the estimate in \eqref{eq:propmapbound}. By the triangle inequality we may write
\begin{equation}\label{eq:propuppfinalstep}
\mathbb{E} \bigg[ \int |T - \id|^2 \dd \rho_0 \bigg] \le \mathbb{E} \bigg[ \int |(T_*  - \id)\circ S_{k_*}|^2 \dd \rho_0 \bigg] + \mathbb{E} \bigg[ \int |S_{k_*} - \id|^2 \dd \rho_0 \bigg]. 
\end{equation}
Let us estimate the two terms in \eqref{eq:propuppfinalstep}. For the first one by definition we may write
\begin{equation}\label{eq:propfirsttermupp}
\begin{split}
\int |(T_*  - \id)\circ S_{k_*}|^2 \dd \rho_0  &\stackrel{\eqref{eq:rhokSk}}{=} \int|T_* - \id|^2 \dd \rho_{k_* }\\
&= \sum_{\text{$Q$ on level $k_*$}} \frac1{N|Q|} N_Q \int_Q |T_* - \id|^2  \\
& \stackrel{\eqref{eq:stopscalTstar}}{\lesssim} \frac{r^2}N \sum_{\text{$Q$ of level $k_*$}} N_Q = r^2.
\end{split}
\end{equation}
Let us now turn to the second term of \eqref{eq:propuppfinalstep}. 

\medskip
{\sc Step 3.1.} We claim that the following estimate holds
\begin{equation}\label{eq:upperformulafinal}
\mathbb{E} \bigg[ \int |S_{k_*} - \id|^2 \dd \rho_0 \bigg] \lesssim \sum_{k: L \ge L_k \ge r}  \bigg(\frac {L_k} r\bigg)^{2-d} r^2.
\end{equation}
We start by noting that by an inductive argument it follows from \eqref{eq:propproofbase} that
\begin{equation}\label{eq:uppformula}
\begin{split} \mathbb{E} \bigg[ \int |S_{k_*} - \id|^2 \dd \rho_0 \bigg]
& \le C \sum_{k: L \ge L_k \ge r} \bigg( \bigg(\frac {L_k} r\bigg)^{2-d} r^2 \prod_{i: L_{k+1} \ge L_i \ge L_{k_*}} \bigg(1 + C \bigg( \frac r {L_i} \bigg)^d \bigg)\bigg) \\
& + \prod_{k: L \ge L_k \ge L_{k_*}} \bigg(1 + C \bigg( \frac r {L_k} \bigg)^d \bigg),
\end{split}
\end{equation}
with the understanding that the product in the first sum is equal to $1$ when $k=k_*$. 
Indeed, iterating for instance two times \eqref{eq:propproofbase} we obtain
\[
\begin{split}
\lefteqn{\mathbb{E} \bigg[ \int |S_{k_*} - \id|^2 \dd \rho_0 \bigg]} \\
& \le C \bigg( \frac{L_{k_*}}r \bigg)^{2-d} r^2 + C \bigg( 1+ C\bigg(\frac r {L_{k_*}}\bigg)^d \bigg) \bigg( \frac{L_{k_*-1}}r \bigg)^{2-d} r^2 \\
& + \bigg( 1+ C\bigg(\frac r {L_{k_*}}\bigg)^d \bigg) \bigg( 1+ C\bigg(\frac r {L_{k_*-1}}\bigg)^d \bigg) \mathbb{E} \bigg[ \int |S_{k_*-2} - \id|^2 \dd \rho_0 \bigg].
\end{split}
\]
By convergence of geometric series we have
\[
\sum_{k: L\ge L_k \ge r} \bigg( \frac r{L_k} \bigg)^{d} 
\lesssim 1
\]
and therefore since $1+x \le e^x$ 
\[
\prod_{k: L \ge L_k \ge L_{k_*}} \bigg(1 + C \bigg( \frac r {L_k} \bigg)^d \bigg)\le \exp\bigg( C \sum_{k: L\ge L_k \ge r} \bigg( \frac r{L_k} \bigg)^{d} \bigg) \lesssim 1.
\]
Combining this with \eqref{eq:uppformula} 
yields \eqref{eq:upperformulafinal}. 

\medskip
{\sc Step 3.2.} Estimating \eqref{eq:upperformulafinal}. We claim that
\begin{equation}\label{eq:propmain}
r^2 \sum_{k: L\ge L_k \ge r} \bigg(\frac{L_k}r\bigg)^{2-d}  \lesssim r^2
\begin{cases}
N & \text{if $d=1$}, \\
\ln N & \text{if $d=2$}, \\
1 & \text{if $d>2$}.
\end{cases}
\end{equation}
By definition of $L_k$, this is a geometric series. 

\medskip
In the one-dimensional case, the large scales $L_k \sim L$ will give the main contribution:
\[
\sum_{k: L\ge L_k \ge r} \frac{L_k}r \sim \frac L r = N.
\]

\medskip
In higher dimensions $d>2$ the small scales $L_k \sim r$ dominate:
\[
\begin{split}
\sum_{k: L\ge L_k \ge r} \bigg( \frac{L_k}r\bigg)^{2-d} \sim 1.
\end{split}
\]

\medskip
In the critical dimension $d=2$ all scales contribute equally and we pick up the number of summands
\[
\sum_{k: L\ge L_k \ge r} \bigg( \frac{L_k}r\bigg)^{2-d} = \# \{ k \, | \, L \ge L_k \ge r \} \stackrel{\eqref{eq:conLk},\eqref{eq:stoppingscale}}{\le} 1+d \log_2 \bigg( \frac L r \bigg) = 1 + \log_2 N
\]

Finally, \eqref{eq:propuppfinalstep}, \eqref{eq:propfirsttermupp}, \eqref{eq:upperformulafinal} and \eqref{eq:propmain} combine to \eqref{eq:propmapbound}.
\end{proof}
\section{Lower Bound}

We now turn to the lower bound of Theorem \ref{thm:AKTasymptotic}. As it will be clear from the proof we do not give the full argument, a complete proof can be found in \cite[Lemma 2.7]{HMO21}.

\medskip
Our argument relies on the dual formulation of the matching problem, that for the $L^1$-cost reads as follows
\[
\frac1N \min_{\sigma \in \mathcal{S}_N} \sum_{n=1}^N |Y_{\sigma(n)} - X_n| = \frac1N\max_{\phi \in \Lip_1} \sum_{n=1}^N \phi(X_n) - \phi(Y_n),
\]
where $\Lip_1$ denotes the set of functions $\phi: [0,L]^d \rightarrow \mathbb{R}$ with Lipschitz constant $1$. Being more precise we will use the previous equation only in the form of
\begin{equation}\label{eq:dualformL1}
\frac1N \min_{\sigma \in \mathcal{S}_N} \sum_{n=1}^N |Y_{\sigma(n)} - X_n|  \ge \frac1N\sup_{\phi} \frac{\sum_{n=1}^N \phi(X_n) - \phi(Y_n)}{\sup_{x \in [0,L]^d} |\nabla \phi(x)|},
\end{equation}
where the supremum is taken over all smooth functions.
Proving the latter is straightforward: Given a smooth function $\phi$, by the mean value theorem we have 
\[
\phi(x) - \phi(y) \le \sup |\nabla \phi|  |x-y|.
\]
Since, for any $\sigma \in S_N$ we have $\sum_{n=1}^N \phi (Y_n) = \sum_{n=1}^N \phi(Y_{\sigma(n)})$ and thus
\[
 \sum_{n=1}^N |Y_{\sigma(n)} - X_n|  \ge \sup_{\phi \in \Lip} \frac{\sum_{n=1}^N \phi(X_n) - \phi(Y_n)}{\sup_{x \in [0,L]^d} |\nabla \phi(x)|}.
\]
Taking the minimum over $S_N$ we obtain \eqref{eq:dualformL1}. The proof of the reverse inequality in \eqref{eq:dualformL1} is more involved, and it will not be discussed in this notes. 
Thanks to \eqref{eq:dualformL1} we can lower bound \eqref{eq:thmasymptotic} via a max problem, hence in order to prove the lower bound \eqref{eq:thmasymptotic} it is sufficient to construct a \emph{``bad''} candidate for the dual formulation.

\begin{proposition}\label{prop:badcandidatedualL1}
Let $X_1, \dots, X_N$ and $r$ be as in Theorem \ref{thm:AKTasymptotic}. Then there exists a differentiable $\phi : [0,L]^d \rightarrow \mathbb{R}$ such that
\begin{equation}\label{eq:badcanddualL1est}
\EE \bigg[\frac1N \sum_{n=1}^N \phi(X_n) - \frac1{L^d} \int_{[0,L]^d} \phi \dd x\bigg] \gtrsim 
\big(\EE[\sup|\nabla \phi|^2]\big)^\frac12
 r  
\begin{cases}
\sqrt{N} & \text{if $d=1$},\\
\sqrt{\ln N} & \text{if $d=2$}, \\
1 & \text{if $d>2$}.
\end{cases}
\end{equation}
\end{proposition}
Note that the candidate $\phi$ that satisfies \eqref{eq:badcanddualL1est} is not \emph{almost surely} in $\Lip_1$. As it will be clear from the proof, the construction is not fine enough to give a $\Lip_1$ map in almost sure sense, but only in \emph{``expectation''}. An almost sure statement will require more involved techniques, see for instance \cite[Section 6]{BobLe}. 
We refer also to \cite[Lemma 2.7]{HMO21} for a construction based on Martingale arguments.

\medskip
We are now in a position to complete the proof of Theorem \ref{thm:AKTasymptotic} given Proposition \ref{prop:badcandidatedualL1}. 
\begin{proof}[Proof of the lower bound of Theorem \ref{thm:AKTasymptotic}]
Note that by the Cauchy-Schwarz inequality, for any $\sigma \in \mathcal{S}_N$ we may write
\begin{equation}\label{eq:prooflow1}
\bigg( \frac1N \sum_{n=1}^N|Y_{\sigma(n)} - X_n|\bigg)^2 \le \frac1N \sum_{n=1}^N |Y_{\sigma(n)} - X_n|^2.
\end{equation}
Furthermore, by the duality formulation \eqref{eq:dualformL1} 
\begin{equation}\label{eq:prooflow2}
 \frac1N \min_{\sigma \in \mathcal{S}_n} \sum_{n=1}^N|Y_{\sigma(n)} - X_n|  \ge \sup_{\phi}
\frac{  \frac1N \sum_{m=1}^N \phi(Y_m) - \frac1N \sum_{n=1}^N \phi(X_n) }{\sup | \nabla \phi|}. 
\end{equation}
Note that $\phi$ depends only on $X_1, \dots, X_N$ and is independent on $Y_1, \dots, Y_N$ which are uniformly distributed on $[0,L]^d$. Thus,
\begin{equation}\label{eq:prooflow3}
\EE\bigg[ \frac1N \sum_{n=1}^N \phi(X_n) - \frac1N \sum_{m=1}^N \phi (Y_m) \bigg] = \EE \bigg[ \frac1N \sum_{n=1}^N \phi(X_n) - \frac1{L^d} \int_{[0,L]^d} \phi \dd x \bigg].
\end{equation}
Hence by the Cauchy-Schwarz inequality w.\,r.\,t. $\EE[\cdot]$ we may write
\begin{equation}\label{eq:prooflow4}
\begin{split}
\lefteqn{\EE \bigg[\bigg( \frac{\frac1N \sum_{n=1}^N \phi(X_n) - \frac1N \sum_{m=1}^N \phi(Y_m)}{\sup | \nabla \phi |}\bigg)^2\bigg]} \\ & \ge \frac{\bigg( \EE \bigg[ \frac1N \sum_{n=1}^N \phi(X_n) - \frac1N \sum_{m=1}^N \phi(Y_m) \bigg] \bigg)^2}{\EE[\sup |\nabla \phi|^2]}.
\end{split}
\end{equation}
Finally, the lower bound \eqref{eq:thmasymptotic} follows by taking the expectation in \eqref{eq:prooflow1} and combining it with \eqref{eq:prooflow2}, \eqref{eq:prooflow3}, \eqref{eq:prooflow4} and \eqref{eq:badcanddualL1est}.
\end{proof}
Finally, we conclude this section with an incomplete proof of Proposition \ref{prop:badcandidatedualL1}.
\begin{proof}[Incomplete proof of Proposition \ref{prop:badcandidatedualL1}]
{\sc Step 1. }Building block for the potential $\phi$.  For $\rho_-\in [0,2]$ let $\rho$ be the measure defined as in \eqref{eq:defrhomeas}. 
Fix a smooth reference function $\zeta$ compactly supported in $(0,1)$ with $\int \zeta \dd x = 1$ and consider
\[
\phi_{\rho_-} (x) = (\rho_- - 1) (\zeta(x) - \zeta(x-1)).
\]
\begin{figure}[ht]%
\center%
\includegraphics[scale=0.75]{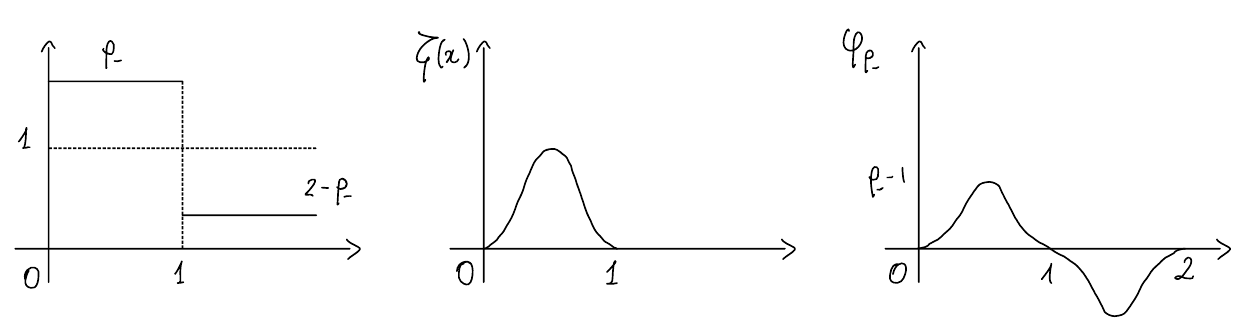}%
\caption{Graph of $\rho_-, \zeta$ and $\phi_{\rho_-}$.}%
\end{figure}%

Then, by construction we have
\begin{equation}\label{eq:eqrho-}
\begin{split}
& \int_{[0,2]} \phi_{\rho_-} \dd \rho = 2(\rho_- -1)^2, \quad \int_{[0,2]} \phi_{\rho_-} \dd x = 0, \ \text{and $\phi_{2-\rho_-} = - \phi_{\rho_-}$}, \\
& \big(\phi_{\rho_-} (x)\big)^2 + \bigg( \frac{\dd \phi_{\rho_-}}{\dd x}(x)\bigg)^2 + \bigg( \frac{\dd^2\phi_{\rho_-}}{\dd x^2}(x) \bigg)^2 \lesssim (\rho_- - 1)^2, \quad \text{for all $x \in [0,2]$}
\end{split}
\end{equation}
\medskip
{\sc Step 2.} Inductive construction. We repeat the construction of boxes of {\sc Step 1} of Proposition \ref{prop:mapupper} and put ourselves into the position of \eqref{eq:mapTk}. Let $\rho_-$ be defined as in \eqref{eq:numbdensleft} and define for $x \in Q$
\begin{equation}\label{eq:concsafe}
\phi_k (x) = L^2_k \phi_{\rho_-}\bigg( \frac{x_1}{L_k}\bigg) \zeta \bigg(\frac{x_2}{L^{(2)}_k}\bigg)\cdot ... \cdot \zeta \bigg(\frac{x_d}{L^{(d)}_k} \bigg),
\end{equation}
where $L^{(i)}_k \in \{ L_k, \frac {L_k}2 \}$ is the length of the side of $Q$ in the coordinate $x_i$, $i=2,\dots, d$.
\begin{figure}[ht]
\center
\includegraphics[scale=0.5]{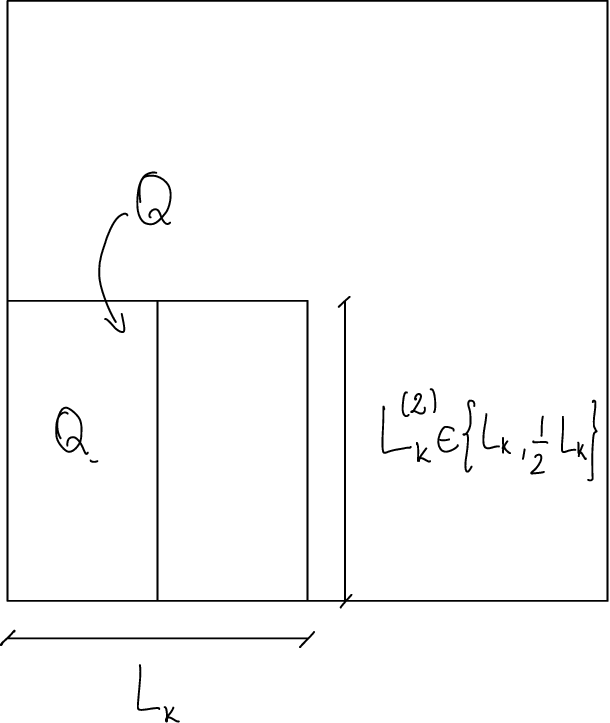}
\caption{Level $k$.}
\end{figure}

\medskip
Define the potential 
\[
\Phi_k := \phi_1 + \dots + \phi_k,
\]
with the understanding that $\Phi_0 = 0$. Note that by definition \eqref{eq:concsafe} $\phi_k = 0$ in a neighborhood of $\partial Q$, therefore $\Phi_k$ is smooth and by \eqref{eq:eqrho-} $\Phi_k$ has vanishing spatial average for any $k=1, \dots, k_*$. Furthermore we claim that on the ``stopping'' scale $k_*$ it holds that
\begin{equation}\label{eq:propcandlow1}
 \EE \bigg[ \frac1N \sum_{n=1}^N \Phi_{k_*} (X_n) - \int \Phi_{k_*} \dd \rho_{k_*} \bigg]  =0. 
\end{equation}
Recall that $X_n$ conditioned on the $\sigma$-algebra $\mathcal{F}_{k_*+1}$ defined in \eqref{eq:sigmaalg} is uniformly distributed on the cube $Q$ that contains $X_n$. In particular we have
\begin{equation}\label{eq:unifphistar}
I(X_n \in Q) \mathbb{E}[\Phi_{k_*}(X_n) \ | \ \mathcal{F}_{k_*+1}]  = I(X_n \in Q) \frac1{|Q|} \int_Q \Phi_{k_*}(x) \,\mathrm{d}x. 
\end{equation}
Note that for any $n \in \{1, \dots, N\}$ there exists at most one cube $Q$ of level $k_*$ that contains $X_n$, thus by the tower property of the conditional expectation we may write
\[
\begin{split}
\mathbb{E}\bigg[\frac 1N \sum_{n=1}^N \Phi_{k_*}(X_n)\bigg] & = \mathbb{E}\bigg[\sum_{\text{$Q$ on level $k_*$}}\sum_{n=1}^N I(X_n \in Q) \frac 1N \mathbb{E}[  \Phi_{k_*}(X_n)\ | \ \mathcal{F}_{k_*+1} ]  \bigg] \\
& \stackrel{\eqref{eq:unifphistar}}{=} \mathbb{E}\bigg[\sum_{\text{$Q$ on level $k_*$}} \sum_{n=1}^N I(X_n \in Q) \frac1{N|Q|}\int_Q \Phi_{k_*}(x)\, \mathrm{d}x \bigg] \\
& = \mathbb{E}\bigg[\sum_{\text{$Q$ on level $k_*$}}  \frac{N_Q}{N|Q|}\int_Q \Phi_{k_*}(x)\, \mathrm{d}x \bigg] \\ 
& \stackrel{\eqref{eq:defrhok}}{=} \mathbb{E}\bigg[\int \Phi_{k_*} \, \mathrm{d}\rho_{k_*}\bigg],
\end{split}
\]
which yields \eqref{eq:propcandlow1}.

\medskip
{\sc Step 3.} We argue that in order to establish \eqref{eq:badcanddualL1est} it is enough to show that for any $k=1, \dots, k_*$
\begin{equation}\label{eq:lowstep3.1}
\EE \bigg[ \int \Phi_k \dd \rho_k\bigg] - \EE\bigg[ \int \Phi_{k-1} \dd \rho_{k-1} \bigg] = \EE \bigg[ \int \phi_k \dd \rho_k \bigg] \gtrsim r^2 \bigg( \frac{L_k}{r} \bigg)^{d-2},
\end{equation}
and
\begin{equation}\label{eq:lowstep3.2}
\sup_{x\in [0,L]^d} \EE [|\nabla \Phi_k (x)|^2] \lesssim r^2 \bigg( \frac{L_k}{r} \bigg)^{d-2} + \sup_{x\in [0,L]^d} \EE [|\nabla \Phi_{k-1} (x)|^2].
\end{equation}
Indeed, once \eqref{eq:lowstep3.1} and \eqref{eq:lowstep3.2} are established by iteration we may write, arguing as in Proposition \ref{prop:mapupper}:
\begin{equation}\label{eq:iteration1low}
\begin{split}
\EE\bigg[ \int \Phi_{k_*} \dd \rho_{k_*} - \int \Phi_{k_*} \dd \rho_0 \bigg] & = \EE\bigg[ \int \Phi_{k_*} \dd \rho_{k_*} - \int \Phi_{0} \dd \rho_0 \bigg]  \\
& \gtrsim r^2 \sum_{k: L \ge L_k \ge r}  \bigg( \frac{L_k}{r} \bigg)^{d-2}  \\
& \gtrsim r^2 
\begin{cases}
N & \text{if $d=1$}, \\
\ln N & \text{if $d=2$},\\
1 & \text{if $d>2$},
\end{cases}
\end{split}
\end{equation}
and
\[
\sup_{x\in [0,L]^d} \EE [|\nabla \Phi_{k_*} (x)|^2] \lesssim \sum_{k: L \ge L_k \ge r}  \bigg( \frac{L_k}{r} \bigg)^{d-2} \lesssim r^2 
\begin{cases}
N & \text{if $d=1$}, \\
\ln N & \text{if $d=2$},\\
1 & \text{if $d>2$}.
\end{cases}
\]
Thus, the proof would be completed once the latter estimate is improved to
\begin{equation}\label{eq:supestimate}
\EE [\sup_{x\in [0,L]^d}  |\nabla \Phi_{k_*} (x)|^2] \lesssim \sum_{k: L \ge L_k \ge r}  \bigg( \frac{L_k}{r} \bigg)^{d-2} \lesssim r^2 
\begin{cases}
N & \text{if $d=1$}, \\
\ln N & \text{if $d=2$},\\
1 & \text{if $d>2$}.
\end{cases}
\end{equation}
This would require to prove 
\[
\sup_{x\in [0,L]^d} \EE [ |\nabla \Phi_{k_*} (x)|^2] \ge \EE [\sup_{x\in [0,L]^d}  |\nabla \Phi_{k_*} (x)|^2],
\]
which in turn would imply the equality as the other direction is straightforward by definition of $\sup$. Since the aim of this notes is to give an introduction to the techniques used in order to study the asymptotics of the matching cost we will not prove \eqref{eq:supestimate}. 
%

%

\medskip
{\sc Step 3.1.} Proof of \eqref{eq:lowstep3.1}. Let $\mathcal{F}_k$ denote the $\sigma$-algebra generated by $N_Q$ of level at most $k-1$ as in \eqref{eq:sigmaalg}. Note that by definition and by the tower property of the conditional expectation we have
\[
\EE \bigg[ \int \Phi_k \dd \rho_k \bigg] - \EE \bigg[ \Phi_{k-1} \dd \rho_{k-1} \bigg] = \EE \bigg[ \int \phi_k \dd  \rho_k \bigg] + \EE \bigg[ \int \Phi_{k-1} (\EE [ \dd \rho_k \ | \ \mathcal{F}_{k-1}] - \dd \rho_k) \bigg].
\]
We claim that $\EE[\dd \rho_k \ | \ \mathcal{F}_{k-1}] = \dd \rho_{k-1}$. Indeed, on a given box $Q$ on level $k-1$ letting $Q_-, Q_+$ denote respectively its left and right half, we may write
\[
\begin{split}
\EE [ \dd \rho_k \ | \ \mathcal{F}_{k-1}] & = \EE\bigg[ \frac{N_{Q_{-/+}}}{N|Q_{-/+}|} \ \bigg| \ \mathcal{F}_{k-1} \bigg] \dd x\\
& = \frac2{N|Q|} \EE[N_{Q_{-/+}} \ | \ \mathcal{F}_{k-1}] \dd x \\
&\stackrel{\eqref{eq:propsigmaalg}}{=} \dd \rho_{k-1}.
\end{split}
\]
We now turn to the estimate on right hand side of \eqref{eq:lowstep3.1}. Note that by definition 
\[
\begin{split}
\int \phi_k \dd \rho_k & = \sum_{\text{$Q$ of level $k-1$}} \frac{N_Q}N \frac1{|Q|} \int_Q \phi_k \frac{\dd \rho_k}{\dd \rho_{k-1}} \dd x \\ 
& = \sum_{\text{$Q$ of level $k-1$}} \frac{N_Q}N L^2_k \int_{[0,2]} \phi_{\rho_-} \dd \rho \\
& \stackrel{\eqref{eq:eqrho-}}{=} \sum_{\text{$Q$ of level $k-1$}} \frac{N_Q}N L^2_k 2 (\rho_- - 1)^2,
\end{split}
\] 
where $\rho_- = \frac{2N_{Q_-}}{N_Q}$. Thus taking the expectation, we obtain
\[
\EE \bigg[ \int \phi_k \dd \rho_k \bigg] = \frac{2L_k^2}N \sum_{\text{$Q$ of level $k-1$}} \EE \bigg[ N_Q \bigg(\frac{2N_{Q_-}}{N_Q} -1\bigg)^2  \bigg] = 2 \frac{L^2_k}N 2^{k-1} \sim \frac{L_k^2}N \bigg(\frac L{L_k}\bigg)^d.
\]

\medskip
{\sc Step 3.2.} Proof of \eqref{eq:lowstep3.2}. By definition of $\Phi_k$ we have the identity
\[
|\nabla \Phi_k|^2 = |\nabla \Phi_{k-1}|^2 + |\nabla \phi_k|^2 + 2 \nabla\Phi_{k-1}\cdot \nabla \phi_k
\]
and thus taking the expectation 
\[
\EE[| \nabla \Phi_k|^2] = \EE[|\nabla \Phi_{k-1}|^2] + \EE[|\nabla \phi_k|^2] + 2\EE[ \nabla\Phi_{k-1} \cdot\EE[\nabla \phi_k \ | \ \mathcal{F}_{k-1}]].
\]
We argue that $\EE[\nabla \phi_k \ | \ \mathcal{F}_{k-1}]=0$. W.\,l.\,o.\,g. we can assume that $d=2$, then for a given box on level $k$ by definition we have
\begin{equation}\label{eq:gradientpotential}
\nabla \phi_k (x) = \bigg( L_k \frac{\dd \phi_{\rho_-}}{\dd x}\bigg( \frac{x_1}{L_k}\bigg)\zeta\bigg( \frac{x_2}{L_k^{(2)}} \bigg), \frac{L^2_k}{L_k^{(2)}} \phi_{\rho_-} \bigg(\frac{x_1}{L_k}\bigg) \frac{\dd \zeta}{\dd x}\bigg(\frac{x_2}{L_k^{(2)}}\bigg) \bigg),
\end{equation}
and thus
\[
\begin{split}
& \EE[\nabla \phi_k (x) \ | \ \mathcal{F}_{k-1}] \\
&= \bigg( L_k \frac{\dd}{\dd x}\EE[\phi_{\rho_-} \ | \ \mathcal{F}_{k-1}] \bigg( \frac{x_1}{L_k}\bigg)\zeta\bigg( \frac{x_2}{L_k^{(2)}} \bigg), \frac{L^2_k}{L_k^{(2)}} \EE[\phi_{\rho_-} \ | \ \mathcal{F}_{k-1}] \bigg(\frac{x_1}{L_k}\bigg) \frac{\dd \zeta}{\dd x}\bigg(\frac{x_2}{L_k^{(2)}}\bigg) \bigg),
\end{split}
\]
so that we need to show $\EE[\phi_{\rho_-} \ | \ \mathcal{F}_{k-1}] = 0$. Recall from {\sc Step 2.2} in Proposition \ref{prop:mapupper} that $\rho_-$ and $2-\rho_-$ have the same law under $\EE[\ \cdot \ | \ \mathcal{F}_{k-1}]$ and that by construction $\phi_{\rho_-} = -\phi_{2-\rho_-}$ so that $\EE[\phi_{\rho_-} \ | \ \mathcal{F}_{k-1}] = - \EE[\phi_{\rho_-} \ | \ \mathcal{F}_{k-1}]$ and thus $\EE[\phi_{\rho_-} \ | \ \mathcal{F}_{k-1}]=0$. Now let us estimate the supremum of the gradient of the potential $\phi_k$. By \eqref{eq:gradientpotential} we have
\[
|\nabla \phi_k (x)|^2 \lesssim L_k^2 \bigg( \phi_{\rho_-}^2 + \bigg( \frac{\dd \phi_{\rho_-}}{\dd x} \bigg)^2\bigg) \bigg( \frac{x_1}{L_k}\bigg) \lesssim L_k^2 (\rho_- - 1)^2 = L_k^2 \bigg(\frac{2N_{Q_-}}{N_Q} - 1\bigg)^2,
\]
and thus by Lemma \ref{lem:lemma1} we may write
\[
\begin{split}
\EE [ | \nabla \phi_k (x)|^2] & \lesssim L_k^2 \EE \bigg[ \frac4{N_Q^2} \EE[ (N_{Q_-}-\frac{N_Q}2)^2 \ | \ \mathcal{F}_{k-1}] \bigg] \\
& \stackrel{\eqref{eq:propsigmaalg}}{=} L_k^2 \EE \bigg[ \frac4{N_Q^2} \frac{N_Q}4  \bigg] \\
& \lesssim L_k^2 \EE \bigg[ I(N_Q\neq 0) \frac1{N_Q}   \bigg] \\
& \stackrel{\ref{eq:lemstoc}}{\lesssim} \frac{L^d}N L_k^{2-d} = r^2 \bigg( \frac{L_k}r \bigg)^{d-2}.
\end{split}
\]
\end{proof}

\section*{Acknowledgment}
The authors would like to thank C. Wagner for his proofreading of an earlier version of the draft.

\bibliographystyle{plain}
\bibliography{OT}
\end{document}